\documentclass[11pt]{amsart}
\usepackage{amssymb, amsmath, latexsym,   bm}

\renewcommand{\baselinestretch}{\baselinestretch}
\renewcommand{\baselinestretch}{1.1}
\numberwithin{equation}{section}

\newtheorem{thm}{Theorem}[section]
\newtheorem{lem}[thm]{Lemma}
\newtheorem{cor}[thm]{Corollary}
\newtheorem{prop}[thm]{Proposition}

\newtheorem{rmk}[thm]{Remark}

\newcommand{\z}{{\mathbb Z}}
\newcommand{\q}{{\mathbb Q}}

\newcommand{\bx}{\bm x}
\newcommand{\by}{\bm y}
\newcommand{\bz}{\bm z}

\newcommand{\bv}{\bm v}
\newcommand{\bu}{\bm u}
\newcommand{\bw}{\bm w}

\newcommand{\lglue}{\bm[\!\![}
\newcommand{\rglue}{]\!\!\bm]}
\newcommand{\cls}{\textnormal{cls}}
\newcommand{\oline}{\overline}

\newcommand{\fs}{\mathfrak s}

\begin{document}

\title{On the exceptional sets of integral quadratic forms}

\author{Wai Kiu Chan}
\address{Department of Mathematics and Computer Science, Wesleyan University, Middletown CT, 06459, USA}
\email{wkchan@wesleyan.edu}

\author{Byeong-Kweon Oh}
\address{Department of Mathematical Sciences and Research Institute of Mathematics,
Seoul National University,
 Seoul 151-747, Korea}
\email{bkoh@snu.ac.kr}
\thanks{This work of the second author was supported by the National Research Foundation of Korea (NRF-2019R1A2C1086347).}

\subjclass[2010]{11E12, 11E25}

\keywords{Integral quadratic forms, exceptional sets, additively indecomposable lattices, root lattices}

\begin{abstract}
A collection $\mathcal S$ of equivalence classes of positive definite integral quadratic forms in $n$ variables is called an $n$-exceptional set if there exists a positive definite integral quadratic form which represents all equivalence classes of positive definite integral quadratic forms in $n$ variables except those in $\mathcal S$.  We show that, among other results, for any given positive integers $m$ and $n$, there is always an $n$-exceptional set of size $m$ and there are only finitely many of them.
\end{abstract}

\maketitle

\section{Introduction}

An integral quadratic form $f(\bx)$ in the $m$ variables $\bx = (x_1, \ldots, x_m)$ is said to represent another integral quadratic form $g(\by)$ in the $n$ variables $\by = (y_1, \ldots, y_n)$ if there exists an $n\times m$ integer matrix $T$ such that
$$f(\by T) = g(\by).$$
One of the fundamental questions in the arithmetic theory of quadratic forms is the representation problem which asks for an effective  determination of the set of $n$-ary quadratic forms that are represented by a given integral quadratic form $f(\bx)$.   This is, of course, the same as deciding which $n$-ary integral quadratic forms that {\em are not} represented by $f(\bx)$.  However, this change of perspective does lead to some interesting problems that have not been investigated before.  For example, can every collection of $n$-ary integral quadratic forms be the exceptional sets of some integral quadratic forms?  Are there only finitely many these exceptional sets of a fixed size?  Can we exhibit such an exceptional set of a given size?  We will answer some of these interesting questions in this paper.

The subsequent  discussion will be conducted in the language of quadratic spaces and lattices.  The readers are referred to \cite{ki} and \cite{om} for any unexplained notations and terminologies.    For simplicity, the quadratic map and its associated bilinear form on any quadratic space will be denoted by $Q$ and $B$, respectively.  The term {\em lattice} always means a finitely generated $\z$-module on a finite dimensional positive definite quadratic space over $\q$.  The scale of a lattice $L$, denoted $\fs(L)$, is the ideal generated by $\{B(\bx, \by): \bx, \by \in L\}$ in $\z$.  We call $L$ an integral lattice if $\fs(L) \subseteq \z$.  The (isometry) class containing $L$ is denoted by $\text{cls}(L)$.

A lattice $M$ is said to be represented by another lattice $L$ if there is a linear map $\sigma: M \longrightarrow L$ such that $Q(\sigma(\bx)) = Q(\bx)$ for all $\bx \in M$.  Such a map is called a representation from $M$ into $L$, which is necessarily injective because $M$ is assumed to be nondegenerate.   Two lattices $L$ and $M$ are isometric if there exists an isometry sending $L$ into $M$.  In this case we will write $L \cong M$.  If $L$ is a lattice and $A$ is one of its Gram matrix, we will write $L \cong A$.   We will often address a positive definite symmetric matrix as a lattice.  The diagonal matrix with entries $a_1, \ldots, a_n$ on its main diagonal will be denoted by $\langle a_1, \ldots, a_n\rangle$.  If $L$ and $M$ are lattices, their orthogonal sum is denoted by $L \perp M$.

For any positive integer $n$, let $\Phi_n$ be the set of classes of integral lattices of rank $n$.  For any integral lattice $L$ of rank $\geq n$,  its {\em $n$-exceptional set}, denoted $\mathcal E_n(L)$, is the set of classes of lattices in $\Phi_n$ not represented by $L$.   A set of classes of lattices of rank $n$ is called an $n$-exceptional set if it is the $n$-exceptional set of some integral lattice.

An $n$-universal lattice is an integral lattice with an empty $n$-exceptional set.  In other words, an $n$-universal lattice is an integral lattice which represents all integral lattices of rank $n$.  More generally, let $\mathcal S$ be a set of classes of integral lattices whose ranks are bounded above by a prescribed constant.  A lattice $L$ is called {\it $\mathcal S$-universal} if it represents all classes in $\mathcal S$.  It was proved in \cite{kko} that there exists a {\it finite} subset $\mathcal S^0 \subset \mathcal S$ such that any $\mathcal S^0$-universal lattice is $\mathcal S$-universal.  Following \cite{lkk}, we call such a finite subset $\mathcal S^0$ an $\mathcal S$-universality criterion set.   An $\mathcal S$-universality criterion set $\mathcal S^0$ is called {\it minimal} if no proper subset of $\mathcal S^0$ is an $\mathcal S$-universality criterion set.  An explicit minimal $\Phi_n$-universality criterion set is known for $n=1,2$ and $8$; see \cite{b}, \cite{kko0}, and \cite{o0}, respectively.   Note that if $\mathcal X$ is a $\Phi_n$-universality criterion set, then $\mathcal E \cap \mathcal X \ne \emptyset$ for any nonempty $n$-exceptional set $\mathcal E$.

Since the classes of rank-1 integral lattices are in one-to-one correspondence with the  positive integers, we will list the elements in an 1-exceptional set or $\Phi_1$-universality criterion set as integers.  For example,  by the 15-Theorem of Conway-Schneeberger \cite{b, c}$, \{1,2,3,5,6,7,10,14,15\}$ is a minimal $\Phi_1$-universality criterion set.

Among all the results obtained in this paper the following two stand out and worth mentioning in this introduction.  For any given positive integers $m$ and $n$:
\begin{itemize}

\item There are only finitely many $n$-exceptional sets of size $m$ (Theorem \ref{finite}).

\item There exists at least one $n$-exceptional set of size $m$ (Theorem \ref{casemn}).
\end{itemize}
They will be proved in Section 2 and Section 5, respectively.  Their proofs rely heavily on root lattices, additively indecomposable lattices, and the lattices which represent them.  All of these will be discussed and analyzed in detail in Section 3 and Section 4.

\section{Finite exceptional sets of integral lattices}


We first present a few results on $n$-exceptional sets of size 1 as a warm-up for the general case.

\begin{prop} \label{1-exceptional}
Let $N$ be an integral lattice of rank $n$. Then $\{\cls(N)\}$ is an $n$-exceptional set if and only if each $\Phi_n$-universality criterion set contains $\cls(N)$.
\end{prop}
\begin{proof}  First, let us assume that $\{\cls(N)\} = \mathcal E_n(L)$ for some integral lattice $L$.  If a $\Phi_n$-universality criterion set $\mathcal X$ does not contain $\cls(N)$, then $L$ represents all classes in $\mathcal X$ and hence is $n$-universal.  This is a contradiction.

Conversely, suppose that every $\Phi_n$-universality criterion set contains the class $\cls(N)$. Let  $\mathcal U$ be a $\Phi_n\setminus\{\cls(N)\}$-universality criterion set.  If every lattice that represents all classes in $\mathcal U$ also represents $\cls(N)$, then $\mathcal U$ itself is a $\Phi_n$-universality criterion set.  This is absurd since $\cls(N) \not \in \mathcal U$.  Therefore, there must be an integral lattice $L$ which represents all classes in $\mathcal U$, but $L$ does not represent  $N$.  This implies that $\mathcal E_n(L) = \{\cls(N)\}$.
\end{proof}

\begin{rmk}
Since every $\Phi_n$-universality criterion set must contain a minimal $\Phi_n$-universality criterion set, Proposition \ref{1-exceptional} can be restated as saying that $\{\cls(N)\}$ is an $n$-exceptional set if and only if every  minimal $\Phi_n$-universality criterion set contains $\cls(N)$.
\end{rmk}

The set $\Phi_1^0: = \{1,2,3,5,6,7,10,14,15\}$ is a $\Phi_1$-universality criterion set by the $15$-Theorem of Conway-Schneeberger \cite{b, c}.  But it is also confirmed in \cite[page 31]{b} that for each of the nine integers in $\Phi_1^0$ there is a (quaternary) integral lattice with that integer as its only exception.   Thus, $\Phi_1^0$ is the unique minimal $\Phi_1$-universality criterion set and the number of 1-exceptional sets of size 1 is 9.  Recently  Barowsky et al \cite{bd} proved that the number of 1-exceptional sets of size 2 is exactly 73.

An integral lattice $N$ is called {\it maximal} (with respect to the scale) if $N\subseteq N'$ for some integral lattice $N' \subseteq \q N$, then $N = N'$.  Any integral lattice with a square-free discriminant is maximal.  Note that if $\{\cls(N)\}$ is an $n$-exceptional set, then $N$ has to be maximal.

An  integral lattice $N$ is called {\it additively indecomposable} if for any representation $\sigma : N \longrightarrow  M_1\perp M_2$, where $M_1$ and $M_2$ are integral lattices, either $\sigma(N)\subseteq M_1$ or $\sigma(N) \subseteq M_2$.  The readers are referred to \cite{p} for some properties of additively indecomposable lattices.  Since unimodular sublattices of an integral lattice must be an orthogonal summand,  any indecomposable unimodular lattice is additively indecomposable.  Bannai \cite{bannai} proved that for every sufficiently large $n$ there must be an indecomposable unimodular lattice of rank $n$.  Therefore, theoretically additively indecomposable integral lattices exist in all sufficiently large dimensions, though in practice these lattices are difficult to find.   For a list of explicit examples of additively indecomposable integral lattices of rank  $n \leq 35$, see \cite{p}.  In Section \ref{additive}, we will exhibit an infinite family of additively indecomposable integral lattices of discriminant 2 (hence maximal) and rank $4(k + 3)$ for every $k \geq 0$.

\begin{lem}
If $N$ is an additively indecomposable maximal integral lattice of rank $n$, then $\{\cls(N)\}$ is an $n$-exceptional set.
\end{lem}

\begin{proof}  Let $\{\cls(K_1),\ldots, \cls(K_\ell)\}$ be a $\Phi_n\setminus \{\cls(N)\}$-universality criterion set.  Then the lattice $L = K_1\perp \cdots \perp K_\ell$ represents all classes in $\Phi_n\setminus \{\cls(N)\}$.  Suppose that $N$ is represented by $L$.  Since $N$ is additively indecomposable, $N$ is represented by $K_i$ for some $i$ which means that $N \cong K_i$ since $N$ is also assumed to be maximal. This contradicts that $\cls(N) \not \in  \{\cls(K_1),\ldots, \cls(K_\ell)\}$.
\end{proof}

In the oppositive direction, additively indecomposable lattices of rank $n$ can also be used to construct infinitely many maximal integral lattices $N$ such that each $\{\cls(N)\}$ is not an  $(n+1)$-exceptional set.  It is a consequence of the following lemma which is a simple observation relating the study of finite exceptional sets to the investigation of universal lattices.

\begin{lem} \label{finite}
If $n > 1$ and $\mathcal E_n(L)$ is finite, then $L$ is $(n-1)$-universal.
\end{lem}
\begin{proof}
Let $N$ be an integral lattice of rank $n-1$.  Since $\mathcal E_n(L)$ is finite, there must be a positive integer $\ell$ such that the lattice $N\perp \langle \ell \rangle$ is represented by $L$.  Then $L$ represents $N$.
\end{proof}

\begin{cor}
Let $N$ be an additively indecomposable integral lattice of rank $n > 1$.  For any positive integer $a$,  $\{\cls(N \perp \langle a \rangle)\}$ is not an $(n+1)$-exceptional set.
\end{cor}
\begin{proof}
Suppose that there is an integral lattice $L$ such that $\mathcal E_{n+1}(L)=\{\cls(N\perp \langle a \rangle)\}$. Since $I_{n+1} \not \cong N\perp \langle a \rangle$, $I_{n+1}$ is represented by $L$ so that $L\cong I_{n+1}\perp L'$ for some integral lattice $L'$.  By Lemma \ref{finite}, $L$ is $n$-universal and hence it represents $N$.  Since $N$ is additively indecomposable, $N$ must be represented by $L'$.  By an old result of Ko \cite{ko}, any additively indecomposable of rank $> 1$ must have rank at least 6.  Therefore, $n \geq 6$ and $I_{n+1}$ represents $a$ by Lagrange's Four-Square Theorem.  This implies that $N \perp \langle a \rangle$ is represented by $L$, which is a contradiction.
\end{proof}

\begin{thm} \label{finiteness}
For any positive integers $m$ and $n$, there are only finitely many $n$-exceptional sets of size $m$.
\end{thm}
\begin{proof}
We fix a $\Phi_n$-universality criterion set $\mathcal X$ and construct families of finite subsets $\mathfrak X_1, \ldots, \mathfrak X_m$ of $\Phi_n$  as follows.  The family $\mathfrak X_1$  is the set containing all nonempty subsets of $\mathcal X$. Suppose that we have constructed the families $\mathfrak X_1, \ldots, \mathfrak X_k$ for $1\leq k < m$.   For every selection of members $\mathcal R_1, \ldots, \mathcal R_k$ in $\mathfrak X_1, \ldots, \mathfrak X_k$ respectively,  fix a choice of an $\Phi_n\setminus (\mathcal R_1 \cup \cdots \cup \mathcal R_k)$-universality criterion set.  The family $\mathfrak X_{k+1}$ is the set containing all these universality criterion sets and their nonempty subsets.  If $\mathcal T$ is one of those sets $\mathcal R_1\cup \cdots \cup \mathcal R_k$,  $\text{E}(\mathcal T)$ denotes the $(\Phi_n\setminus \mathcal T)$-universality criterion set chosen in the process.  Let $\overline{\mathfrak X}$ be the union $\mathfrak X_1 \cup \cdots \cup \mathfrak X_m$, which is a finite set independent of any $n$-exceptional sets.

Let $\mathcal E$ be an $n$-exceptional set of size $m$,  and $L$ be an integral lattice such that $\mathcal E_n(L)=\mathcal E$.   If $\mathcal X\cap \mathcal E$ is the empty set, then $L$ would be $n$-universal which is a contradiction.  Thus,  $\mathcal R_1: = \mathcal X \cap \mathcal E$ is a nonempty subset of $\mathcal X$ of size $m_1 \le m$.  If $m = m_1$, then $\mathcal E = \mathcal R_1 \in \mathfrak X_1 \subseteq \overline{\mathfrak X}$.  Suppose that $m_1 < m$, and let us assume that  $\text{E}(\mathcal R_1)\cap \mathcal E = \emptyset$.  Then $L$ represents all classes in $\Phi_n\setminus \mathcal R_1$. In particular, $L$ represents all classes in $\mathcal E\setminus \mathcal R_1$.   This is impossible since any class in $\mathcal E$ by definition is not represented by $L$.

So, $\mathcal R_2: = \text{E}(\mathcal R_1)\cap \mathcal E$ is a nonempty subset of $\text{E}(\mathcal R_1)$ of size $m_2 \leq m - m_1$.  If $m_2 = m - m_1$, then $\mathcal E = \mathcal R_1 \cup \mathcal R_2$ which is contained in $\mathfrak X_1 \cup \mathfrak X_2 \subseteq \overline{\mathfrak X}$.  Suppose $m_2 < m - m_1$.  If $\text{E}(\mathcal R_1 \cup \mathcal R_2)\cap \mathcal E = \emptyset$, then $L$ would represent all classes in $\mathcal E\setminus (\mathcal R_1 \cup \mathcal R_2)$ which is not possible.  So, we may continue this argument and decompose $\mathcal E$ into the union of $\mathcal R_1, \ldots, \mathcal R_k$ with $\mathcal R_j \in \mathfrak X_j$ for $j = 1, \ldots, k \leq m$.  Then $\mathcal E$ is contained in the finite set $\overline{\mathfrak X}$.  This proves the theorem.
\end{proof}

\section{The root lattices}

Suppose that $L_1, \ldots, L_t$ are integral lattices.  For $i = 1, \ldots, t$, let $\bx_i$ be a vector in $L_i^\#$.  Define
\begin{equation} \label{gluesum}
L_1\cdots L_t\left[\bx_1\cdots \bx_t\right]: = (L_1 \perp \cdots \perp L_t) + \z[\bx_1 + \cdots + \bx_t].
\end{equation}
This lattice is integral if and only if $Q(\bx_1 + \cdots + \bx_t)$ is an integer.  If $L_i = \z[\bz_i] \cong \langle a \rangle$ for some integer $a$ and $\bx_i = \frac{\bz_i}{m}$, then in the notation $L_1\cdots L_t\left[\bx_1\cdots \bx_t\right]$  we will use ``$a$" instead of $L_i$ and replace $\bx_i$ by $\frac{1}{m}$.

An integral lattice is a root lattice if it is spanned by its roots, i.e. vectors $\bx$ such that $Q(\bx) \leq 2$.  The indecomposable root lattices are $I_1$, $A_n$, $D_n (n \geq 4)$, $E_6, E_7,$ and $E_8$.  The readers can find a detailed analysis of these root lattices in \cite[Chapter 4]{cs}.  We will only review some of the properties useful for later discussion.  In \eqref{gluesum}, if $L_i$ is a root lattice, then $\bx_i$ is chosen from the set of glue vectors defined in \cite[Chapter 4]{cs}.  In this case, each glue vector is denoted by $\lglue\ell\rglue$ for some nonnegative integer $\ell$ and we simply use ``$\ell$" in \eqref{gluesum}.  These glue vectors for root lattices have the property that they are vectors of the shortest length in their cosets.

For $n \geq 1$,  the root lattice $A_n$ is
$$A_n = \{(a_0, a_1, \ldots, a_n) \in \z^{n+1} : a_0 + \cdots + a_n = 0\}$$
which is viewed as a sublattice in $\z^{n+1}$.  It is an integral lattice of rank $n$ and discriminant $n + 1$.
Its glue vectors are
\begin{equation} \label{gluean}
\lglue i\rglue  = \left(\frac{i}{n+1}, \ldots, \frac{i}{n+1}, \frac{-j}{n+1}, \ldots, \frac{-j}{n+1} \right),
\end{equation}
with $j$ components equal to $i/(n+1)$, and $i$ components equal to $-j/(n+1)$, where $i + j = n+1$ and $0 \leq i \leq n$.  As an example of illustrating \eqref{gluesum}, $A_n\, a \left[i\, \frac{1}{d}\right]$ is the lattice
$$(A_n\perp \z[\bz]) + \z\left[\lglue i\rglue + \frac{\bz}{d}\right]$$
where $\bz$ is a vector orthogonal to $A_n$ such that $Q(\bz) = a$.  Another example is $A_{11}\, A_5[2\, 2]$, which is the lattice $(A_{11}\perp A_5) + \z\left[\lglue 2\rglue + \lglue 2\rglue\right]$ of rank 17.  Note that the first ``$\lglue 2\rglue$" is the glue vector for $A_{11}$ and the second one is for $A_5$.

For any $n \geq 4$, the root lattice $D_n$ is
$$D_n = \{(a_1, \ldots, a_n)\in \z^n : a_1 + \cdots + a_n \equiv 0 \mod 2 \}.$$
It is an integral lattice of rank $n$ and discriminant 4.  Its glue vectors are
\begin{equation} \label{gluedn}
\begin{split}
\lglue 0\rglue & =  (0, \ldots, 0),\\
\lglue 1\rglue & =  (1/2, \ldots, 1/2),\\
\lglue 2\rglue & =  (0, \ldots,0, 1),\\
\lglue 3\rglue & =  (1/2, \ldots, 1/2, -1/2).\\
\end{split}
\end{equation}

For $n \not \equiv 0$ mod 4, $I_n$ is the only integral lattice of rank $n$ which properly contains $D_n$.  However, if $n \equiv 0$ mod 4, there is another integral lattice of rank $n$ besides $I_n$ which properly contains $D_n$.  This lattice is necessarily unimodular and is built by adjoining the glue vector $\lglue 1\rglue$ to $D_n$, which we will denote by $D_n[1]$.

\begin{lem} \label{dnrep}
If $L$ is an integral lattice which represents $D_n$ imprimitively, then $L$ must be isometric to $I_n\perp N$ or $D_n[1] \perp N$ for some integral lattice $N$.
\end{lem}
\begin{proof}
This is clear.
\end{proof}

\begin{lem} \label{andn}
Let $L$ be an integral lattice of rank $n + 1$.
\begin{enumerate}
\item If $L$ primitively represents $D_n$, then $L$ is isometric to
$$D_{n}\perp \langle a \rangle, \quad D_{n}\, (16a - 4n)\left[1\, \frac{1}{4}\right], \quad \mbox{ or } \quad D_{n}\, 4a\left[2\, \frac{1}{2}\right]$$

\item If $L$ primitively represents $A_n$, then $L$ is isometric to
$$A_n\, k(n+1)\left[i\, \frac{1}{n+1}\right]$$
for $0 \leq i \leq \lfloor \frac{n+1}{2}\rfloor$ and some integer $k$ such that $k \equiv i^2 \mod n+1$.
\end{enumerate}
\end{lem}
\begin{proof}
Suppose that $M$ is an integral lattice of rank $n$ which is a primitive sublattice of an integral lattice of rank $n + 1$.  Then $L = M + \z[\bx]$ for some $\bx \in L$.  Let $\bz$ be a vector in $L$ which generates the orthogonal complement of $M$ in $L$.  Then $\bz = \by + f\bx$ for some $\by \in M$ and an integer $f$, which can be assumed to be positive, so that
$$L = (M \perp \z[\bz]) + \z\left[-\frac{\by}{f} + \frac{\bz}{f}\right].$$
It is clear that $\by': = -\frac{\by}{f} \in M^\#$ and hence we may choose $\by'$ from a fixed set of coset representatives of $M^\#/M$.  Moreover, since $\by \in M$, the order of the coset containing $\by'$ as an element of $M^\#/M$ must divide $f$.  But the vector $\bz$ is primitive in $L$; thus $f$ must be the order of the coset $\by' + M$ in $M^\#/M$.
\bigskip

\noindent (1) When $M = D_n$, the vector $\by'$ can be chosen to be $\lglue i\rglue$, $0 \leq i \leq 3$.  Since there is an isometry of $D_n$ which exchanges $\lglue 1\rglue$ and $\lglue 3\rglue$, the choices of $\by'$ can be trimmed down to $\lglue 0\rglue, \lglue 1\rglue$, and $\lglue 2\rglue$.  If $f = 1$, then $\by' = \lglue 0\rglue \in D_n$ and hence $L \cong D_n \perp \langle a \rangle$ for some $a \geq 1$. If $f = 4$, then $\by' = \lglue 1\rglue$.  Suppose that $Q(\lglue 1\rglue + \frac{\bz}{4}) = a$, which must be $> Q(\lglue 1\rglue) = \frac{n}{4}$.  Then $\frac{n}{4} + \frac{Q(\bz)}{16} = a$, meaning that $Q(\bz)= 16a - 4n$ and $L \cong D_n\, (16a - 4n)\left[1\, \frac{1}{4}\right]$.  Finally, if $f = 2$, then $\by'$ could be $\lglue 2\rglue$ or $\lglue 1\rglue$ (only when $n \equiv 0$ mod 2).  If $\by' = \lglue 2\rglue$, then one can show that $L \cong D_n\, 4a\left[2\, \frac{1}{2}\right]$ for some $a \geq 1$.  But if $\by' = \lglue 1\rglue$, then $Q(\bz) = 4a - n$ and $L$ is isometric to
\begin{eqnarray*}
(D_n\perp \z[\bz]) + \z\left[\lglue 1\rglue + \frac{\bz}{2} \right] & = & (D_n \perp \z[2\bz]) + \z\left[\lglue 1 \rglue + \frac{2\bz}{4}\right] \\
 & = & D_n\, (16a - 4n)\left[1\, \frac{1}{4}\right].
\end{eqnarray*}
\bigskip

\noindent (2)  When $M = A_n$, then $\by'$ can be chosen to be $\lglue i\rglue$, $0 \leq i \leq n$.  However, since $\lglue i\rglue = -\lglue j\rglue$ if $i + j = n + 1$, we may assume that $\by' = \lglue i\rglue$ with $0 \leq i \leq \lfloor \frac{n+1}{2} \rfloor$.  If the order of $\lglue i\rglue$ is $f$, let $\bz': = \frac{n+1}{f}\bz$. Then $L$ is isometric to
$$(A_n\perp \z[\bz]) + \z\left[\lglue i\rglue+ \frac{\frac{n+1}{f}\bz}{n+1}\right] = (A_n\perp \z[\bz']) + \z\left[\lglue i\rglue + \frac{\bz'}{n+1}\right].$$
If $Q(\lglue i\rglue + \frac{\bz'}{n+1}) = a$, then $Q(\bz') = (n+1)(a(n+1) - i(n+1) + i^2)$.  This proves (2).
\end{proof}

As a consequence of this lemma, for small values of $n$ it is not too difficult to determine the complete list of integral lattices of rank $n + 1$ that represent $A_n$.  When $n = 2$, any representation of $A_2$ into an integral lattice must be a primitive representation.  So, an integral lattice which represents $A_2$ must be isometric to
\begin{equation} \label{a2}
A_2 \perp \langle a \rangle \quad \mbox{ or }  \quad A_2\, (9a - 6) \left[1\,\frac{1}{3}\right] \cong \begin{pmatrix} 2 & -1 & 0\\ -1 & 2 & -1\\ 0 & -1 & a \end{pmatrix}
\end{equation}
for some suitable integer $a$.  By considering the discriminant, we see that the lattices in \eqref{a2} are mutually non-isometric.  Note that $A_2\, 3\left[1\, \frac{1}{3}\right] \cong I_3$ and $A_2 \, 12 \left[1\, \frac{1}{3}\right] \cong A_3$.

For $n = 3$,   it is easy to see that $I_3$ is the only maximal integral lattice of rank 3 that represents $A_3$.   Therefore, an integral lattice of rank 4 which represents $A_3$ must be isometric to
\begin{equation} \label{a3}
I_3 \perp \langle b \rangle, \quad A_3 \perp \langle b \rangle, \quad A_3\, 4b\left[2\, \frac{1}{2}\right],\quad \mbox{ or }\quad A_3\, (16b - 12)\left[1\, \frac{1}{4}\right]
\end{equation}
for some suitable integer $b$.  Note that $A_3\, 4 \left[2\, \frac{1}{2}\right] \cong D_4$, $A_3\, 4\left[1\, \frac{1}{4}\right] \cong I_4$, and $A_3\, 20\left[1\, \frac{1}{4}\right] \cong A_4$.

The case for $A_4$ is similar to that for $A_2$ since the discriminant of $A_4$ is also a prime.  An integral lattice which represents $A_4$ must be isometric to
\begin{equation} \label{a4}
A_4 \perp \langle b \rangle, \quad M_c: = A_4\, (25c - 20)\left[1\, \frac{1}{5}\right], \quad \mbox{ or } \quad K_d: = A_4\, (25d - 5)\left[2\,\frac{1}{5}\right],
\end{equation}
for some suitable integers $b$, $c$, and $d$.  Note that $M_1 \cong I_5$, $M_2 \cong A_5$, and $K_1 \cong D_5$.

\begin{lem} \label{a4rep}
The smallest positive integer $c$ for which $M_{c}$ is represented by $K_d$ is $c = 4d$, and the smallest positive integer $d$ for which $K_{d}$ is represented by $M_c$ is $d = 4c - 3$.
\end{lem}
\begin{proof}
We only give the proof of the first assertion; the second assertion can be proved in the same way.

The discriminants of $K_d$ and $M_{c}$ are $5d - 1$ and $5c - 4$, respectively. If $M_{c}$ is represented by $K_d$, then $(5d - 1)k^2 = 5c - 4$ for some positive integer $k$.  Clearly,  $k$ cannot be 1 and hence $5c - 4 \geq (5d - 1)4$, meaning that $c \geq 4d$.  So, it suffices to show that $M_{4d}$ is represented by $K_d$.

Let $\bz$ be a vector in $K_d$ orthogonal to $A_4$ such that $Q(\bz) = 25d - 5$.  Then, according to \eqref{a4}, we may assume that $K_d$ is the lattice $(A_4 \perp \z[\bz]) + \z\left[\lglue 2\rglue + \frac{\bz}{5}\right]$.  Then $\lglue 4\rglue + \frac{2\bz}{5}$ is in $K_d$, thus $M_{4d} \cong (A_4\perp \z[-2\bz]) + \z\left[\lglue 1\rglue + \frac{-2\bz}{5}\right]$ is in $K_d$.
\end{proof}

\begin{lem} \label{anrep}
Suppose that both $A_n\, k(n+1)\left[i\, \frac{1}{n+1}\right]$ and $A_n\, \ell(n+1)\left[j\, \frac{1}{n+1}\right]$, $0 \leq i, j \leq \lfloor \frac{n+1}{2} \rfloor$,  have $A_n$ as their root sublattices.  Then $A_n\, k(n+1)\left[i\, \frac{1}{n+1}\right]$ is represented by $A_n\, \ell(n+1)\left[j\, \frac{1}{n+1}\right]$ if and only if there exists a positive integer $t$ such that
$$jt\equiv \pm i \mod n+1 \quad \mbox{ and } \quad k = \ell t^2.$$
\end{lem}
\begin{proof}
Let $\bz$ be a vector orthogonal to $A_n$ such that $Q(\bz) = \ell(n+1)$ and
$$A_n\, \ell(n+1)\left[j\, \frac{1}{n+1}\right] = (A_n\perp \z[\bz]) + \z\left[\lglue j\rglue + \frac{\bz}{n+1}\right].$$
This lattice contains
$$(A_n + \z[t\bz]) + \z\left[t\lglue j\rglue + \frac{t\bz}{n+1}\right]$$
for any positive integer $t$.  This proves the ``if" part of the assertion of the lemma.

For the ``only if" part, let $\bx$ be a vector orthogonal to $A_n$ such that $Q(\bx) = k(n+1)$ and
$$A_n\, k(n+1)\left[i\, \frac{1}{n}\right] = (A_n\perp \z[\bx]) + \z\left[\lglue i\rglue + \frac{\bx}{n+1}\right].$$

Suppose that $\sigma$ is a representation sending $A_n\, k(n+1)\left[i\, \frac{1}{n}\right]$ into $A_n\, \ell(n+1)\left[j\, \frac{1}{n+1}\right]$.
Let $t$ be the group index of the image of $\sigma$ in $A_n\, \ell(n+1)\left[j\, \frac{1}{n+1}\right]$.  By comparing the discriminants, we have $\ell t^2 = k$.  Under the assumption on the root sublattices and the fact that $A_n$ is primitive in both $A_n\, \ell(n+1)\left[j\, \frac{1}{n+1}\right]$ and $A_n\, k(n+1)\left[i\, \frac{1}{n}\right]$, the restriction of $\sigma$ on $A_n$ becomes an isometry of $A_n$.  Moreover,
$$\sigma\left(\lglue i\rglue + \frac{\bx}{n+1}\right) \equiv t\lglue j\rglue + \frac{t\bz}{n+1} \mod A_n,$$
and hence $\sigma(\lglue i\rglue) \equiv t\lglue j\rglue$ mod $A_n$.  Let $i'$ be the integer between 0 and $n$ such that $jt\equiv i'$ mod $n + 1$.  The glue vector $\lglue i'\rglue$ has the shortest length among all vectors in its coset.   Therefore, $Q(\lglue i\rglue) = Q(\sigma(\lglue i \rglue)) \geq Q(\lglue i'\rglue)$.  By switching the roles of $i$ and $i'$ we deduce that $Q(\lglue i'\rglue) = Q(\lglue i\rglue)$.  This implies either $i = i'$ or $i + i' = n+1$, that is, $i \equiv \pm i'$ mod $n + 1$ which is what we need to prove.
\end{proof}

The root lattice $E_8$ is the unique (up to isometry) even unimodular lattice of rank 8.  If we use the even coordinate system for $E_8$ (see \cite[Page 120]{cs}), then the root lattice $E_7$ can be realized as
$$E_7: = \{(a_1, \ldots, a_8) \in E_8: a_1 + \cdots + a_8 = 0\}.$$
Its discriminant is 2.  Its glue vectors are
\begin{equation} \label{gluee7}
\begin{split}
\lglue 0\rglue & =  (0, 0,0,0,0,0,0, 0),\\
\lglue 1\rglue & =  (1/4, 1/4, 1/4, 1/4, 1/4, 1/4, -3/4, -3/4).\\
\end{split}
\end{equation}
It is known \cite{ko, p} that $E_7$ is additively indecomposable.

\begin{lem} \label{a8rep}
If $L$ is an integral lattice which represents $A_8$ imprimitively, then $L$ is isometric to $E_8 \perp N$ for some integral lattice $N$.
\end{lem}
\begin{proof}
Suppose that $\sigma : A_8 \longrightarrow L$ is a representation such that $\sigma(A_8)$ is not primitive in $L$.  Then $\q\sigma(A_8)\cap L$ is a unimodular lattice which must be generated by $A_8$ and its glue vector $\lglue 3 \rglue$.  Since $Q(\lglue 3 \rglue) = 2$, this unimodular lattice must be even and hence isometric to $E_8$.
\end{proof}

\section{Additively indecomposable lattices} \label{additive}

We resume the discussion on additively indecomposable lattices in this section.

\begin{lem} \label{min}
If $L$ is an additively indecomposable integral lattice, then $\min(L^\#) > 1$.
\end{lem}
\begin{proof}
The quadratic form on an additively indecomposable integral lattice $L$ is necessarily a ``block form" (see \cite[Definition (II.1)]{p}) which, by \cite[Corollary (II.5)]{p},  is equivalent to the condition $\min(L^\#) > 1$.
\end{proof}

\begin{lem}\label{mindual}
Let $L$ be an integral lattice of squarefree discriminant $d$.  Then for any $\bv \in L^\#\setminus L$, $Q(\bv) \in \frac{1}{d}\z\setminus \z$.
\end{lem}
\begin{proof}
For any $\bv \in L^\#\setminus L$, $d\bv \in L$ and hence $B(\bv, d\bv) \in \z$.  So, $Q(\bv) \in \frac{1}{d}\z$.  If $Q(\bv) \in \z$, then $L + \z[\bv]$ is an integral lattice containing $L$ as a proper sublattice.  This is impossible as $d$ is squarefree.
\end{proof}

\begin{prop} \label{addindecomposable1}
The lattice $L_{12}: = E_7\, A_5[1\, 3]$ is an additively indecomposable integral lattice with discriminant $3$ and $\min(L_{12}^\#) = \frac{4}{3}$.  Any integral lattice of rank $13$ which represents $L_{12}$ is isometric to either $L_{12}\perp \langle a \rangle$ or $L_{12}\, (9a - 3)\left [\bv\, \frac{1}{3}\right]$, where $a$ is some suitable positive integer and $\bv$ is a suitable minimal vector of $L_{12}^\#$.
\end{prop}
\begin{proof}
By \cite[Table 1]{cs1}, $L_{12}$ is the unique (up to isometry) indecomposable integral lattice of rank 12 and discriminant 3.  Since every additively indecomposable lattice must be indecomposable, the additively indecomposable integral lattice of rank 12 and discriminant 3 appeared in \cite[Example (III.3)]{p} (whose symbol is $1^{12-8}, 2^7; 6$) must be $L_{12}$.  We also know from \cite[Example (III.3)]{p} that $\min(L^\#) =  1 + \frac{1}{3} = \frac{4}{3}$.

Let $\bx_1$ be the glue vector $\lglue 1\rglue$ of $E_7$, and $\bx_2$ be the glue vector $\lglue 3\rglue$ of $A_5$.     Let $\bv$ be the glue vector $\lglue 2\rglue$ of $A_5$. Since $B(\bx_1 + \bx_2, \bv) = B(\bx_2, \bv) = 1$ and $Q(\bv) = \frac{4}{3}$, $\bv$ is a minimal vector of $L_{12}^\#$.

In addition, $\{0, \bv, 2\bv\}$ is a complete set of representatives of $L_{12}^\#/L_{12}$.  Since the discriminant of $L_{12}$ is 3, any representation of $L_{12}$ into an integral lattice must be primitive.  Then a straightforward calculation shows that any integral lattice which represents $L_{12}$ must be isometric to
$$L_{12} \perp \langle a \rangle, \quad L_{12}\, (9a - 3)\left[\bv\, \frac{1}{3}\right], \quad \mbox{ or } \quad L_{12}\, (9a - 3)\left[2\bv\, \frac{1}{3}\right]$$
where $a$ is a suitable positive integer.  Since $\bv \equiv -2\bv$ mod $L_{12}$, the last two lattices are isometric via the symmetry with respect to the orthogonal complement of $L_{12}$.
\end{proof}

\begin{prop} \label{addindecomposable2}
The lattice $L_{16}: = A_{11}\, A_5\left[2\,2\right]$ is an additively indecomposable integral  lattice with discriminant $2$ and $\min(L_{16}^\#) = \frac{3}{2}$.  Any integral lattice of rank $17$ which represents $L_{16}$ is isometric to either $L_{16}\perp \langle b\rangle$ or $L_{16}\, (4b-2)\left[\bw\, \frac{1}{2}\right]$, where $b$ is a suitable positive integer and $\bw$ is a suitable minimal vector of $L_{16}^\#$.
\end{prop}
\begin{proof}
By \cite[Table 2]{cs1}, $L_{16}$ is an indecomposable integral lattice of rank 16 and discriminant 2.    However, there is another lattice in its genus which is also indecomposable.  The sublattice of $L_{16}$ generated by its roots is $A_{11} \perp A_5$.    Let $\bw$ be the glue vector $\lglue 3\rglue$ of $A_5$.  It is straightforward to check that $\bw$ is in $L^\#$ and $Q(\bw) = \frac{3}{2}$.  If $\bw$ is not a minimal vector of $L^\#$, then either $L_{16}$ represents 1 or by Lemma \ref{mindual} $L_{16}^\#$ must have a vector $\by$ such that $Q(\by) = \frac{1}{2}$.   The first possibility is absurd since $L_{16}$ is indecomposable.  As for the second possibility, $2\by$ would be a root of $L_{16}$.  There must be another root $\bv$ in $A_{11}$ or $A_5$ such that $B(2\by, \bv) = 1$.  But this is impossible as $\by \in L_{16}^\#$.  Thus,  $\min(L^\#)$ must be equal to $\frac{3}{2}$.

Let $\bx_1$ be the glue vector $\lglue 2\rglue$ of $A_{11}$ and $\bx_2$ be the glue vector $\lglue 2\rglue$ of $A_5$.  Then $Q(\bx_1 + \bx_2) = 3$, implying that  $L_{16}$ is generated by vectors $\bx$ with $Q(\bx) \leq 3$.  By \cite[Proposition (III.1)]{p}, $L_{16}$ is additively indecomposable.  The rest of the proof is similar to that of Proposition \ref{addindecomposable1} and we leave it to the readers.
\end{proof}

Let $M = A_{13}\, 7\left[4\, \frac{1}{7}\right]$.  By \cite[Table 1]{cs1}, $M$ is the unique indecomposable integral lattice of rank 14 and discriminant 2.   So, it must be the additively indecomposable integral lattice of the same rank and discriminant in \cite[Example (III.3)]{p} (with symbol $1^{14-1}; 4$).  The minimum of $M^\#$ is $1 + 2^{-1} = \frac{3}{2}$.  Let $\bz$ be a vector which generates the orthogonal complement of $A_{13}$ in $M$.  It is easy to verify that $\bu: = \lglue 1\rglue + \frac{2}{7}\bz$ is in $M^\#$ and $Q(\bu) = \frac{3}{2}$.  So, $\bu$ is a minimal vector of $M^\#$ and $M^\# = M \cup (\bu + M)$.

For every $k \geq 2$, let
\begin{equation} \label{m4n+3}
M_{4(k+3)}: = M\, D_{4k-2}[\bu\, 1].
\end{equation}
It is easy to check that $M_{4(k+3)}$ is an integral lattice of rank $4(k+3)$ and discriminant 2.

\begin{lem} \label{add}
Let $L_1$ and $L_2$ be integral lattices which are either additively indecomposable or indecomposable root lattices.  Suppose that $L := L_1L_2[\bx_1\,\bx_2]$ is an integral lattice.  If either $Q(\bx_1)$ or $Q(\bx_2)$ is not an integer, then $L$ is either additively indecomposable or represented by a sum of squares.
\end{lem}
\begin{proof}
Suppose that $\sigma : L \longrightarrow M_1\perp M_2$ is a representation of $L$ into the orthogonal sum of two integral lattices $M_1$ and $M_2$.  For $j = 1, 2$, $M_i$ takes the form $I_{m_j} \perp M_j'$ with $\min(M_j') \geq 2$ if $M_j'$ is nonzero.  The hypothesis of the lemma implies that $\sigma(L_i)$ is contained in $M_1'$, $M_2'$, or $I_{m_1 + m_2}$ for $i = 1, 2$.

If both $\sigma(L_1)$ and $\sigma(L_2)$ are inside of $I_{m_1 + m_2}$,  then $\sigma(\bx_1 + \bx_2)$ must also be in $I_{m_1 + m_2}$ and hence $\sigma(L)$ is represented by a sum of squares.

Suppose that one of $\sigma(L_1)$ or $\sigma(L_2)$ is not contained in $I_{m_1 +m_2}$.  Without loss of generality, let us assume that $\sigma(L_1) \subseteq M_1'$.  Suppose that $\sigma(L_2)$ is contained in either $M_2'$ or $I_{m_1+m_2}$.  Since $\sigma(\bx_1 + \bx_2) \in (\q M_1' \perp \q N) \cap (M_1 \perp M_2) = M_1' \perp N$ with $N = M_2'$ or $I_{m_1+m_2}$,  therefore $\sigma(\bx_1) \in M_1'$ and $\sigma(\bx_2) \in N$ which is not possible because one of $Q(\bx_1)$ and $Q(\bx_2)$ is not an integer.  Therefore, $\sigma(L_2)$ is also inside of $M_1'$ which means that $\sigma(L) \subseteq M_1$ and $L$ is additively indecomposable.
\end{proof}

\begin{prop} \label{mn}
For every integer $k \geq 2$, the integral lattice $M_{4(k+3)}$ defined in \eqref{m4n+3} is additively indecomposable.  The minimum of $M_{4(k+3)}^\#\setminus M_{4(k+3)}$ is  $\frac{5}{2}$ but the minimum of $M_{4(k+3)}$ is $2$.
\end{prop}
\begin{proof}
Since $M$ is additively indecomposable, it cannot be represented by any sum of squares.  Therefore, $M_{4(k+3)}$ cannot be represented by any sum of squares.  It must then be additively indecomposable by Lemma \ref{add}.  As a result, the minimum of $M_{4(k+3)}^\#$ is $> 1$ by Lemma \ref{min}.  It is easy to verify that the vector $\bu + \lglue 2\rglue$ is in $M_{4(k+3)}^\#\setminus M_{4(k+3)}$ and $Q(\bu + \lglue 2\rglue) = \frac{5}{2}$.
To complete the proof, it suffices to show that $M_{4(k+3)}^\#\setminus M_{4(k+3)}$ does not have any vector $\bx$ with $Q(\bx) = \frac{3}{2}$.   Assume on the contrary that there were indeed such a vector $\bx$.  Then, $\bx = \by + \bz$ where $\by \in M^\#$ and $\bz \in D_{4k-2}^\#$, and
$$B\left(\by + \bz, \bu + \lglue 1\rglue\right) \in \z.$$
Since $\min(M^\#) \geq \frac{3}{2}$, we have $\by = 0$ or $Q(\by) = \frac{3}{2}$.  If $\by = 0$, then $B(\bz, \lglue 1\rglue) \in \z$ which means that $\bz$ is in $D_{4k-2}$ or $\lglue 3\rglue + D_{4k-2}$.  But then $Q(\bz) \neq \frac{3}{2}$ which is not possible.  If $Q(\by) = \frac{3}{2}$, then $\bz = 0$ and $\bx = \by$.  We know that $\bx \not \in M$ because $Q(\bx) \not \in \z$.  Then, $\bx \in \bu + M$ and
$$B\left(\bx, \bu + \lglue 1\rglue\right) = B\left(\bx, \bu + \lglue 1\rglue\right) = B(\bx, \bu) \in Q(\bu) + \z \nsubseteq \z$$
which is impossible again.  This completes the proof.
\end{proof}

\section{Existence of $n$-exceptional sets of arbitrary sizes}

This section is devoted to proving the following theorem.

\begin{thm} \label{casemn}
For any positive integers $m$ and $n$, there exists an integral lattice whose $n$-exceptional set has size $m$.
\end{thm}

Although the proof of this theorem is divided into several cases, there is a common thread to the arguments in all these different cases which we will explain in the following.

Let $\mathcal S$ be a subset of $\Phi_n$, and let $\{\cls(N_1), \ldots, \cls(N_\ell)\}$ be an $\mathcal S$-universality criterion set.  We call the lattice $\mathfrak N: = N_1\perp \cdots \perp N_\ell$ a {\em universal hull} for $\mathcal S$, which is an $\mathcal S$-universal integral lattice.  Denote by $\overline{\mathfrak N}$ the orthogonal sum of the indecomposable components of $\mathfrak N$ that are not isometric to $I_1$.
Note that the lattice $I_n \perp \oline{\mathfrak N}$ is also an $\mathcal S$-universal lattice.

Depending on $n$, we will consider a universal hull $\mathfrak N$ (or $I_n\perp \oline{\mathfrak N}$)  for some specific $\mathcal S$.  For example, we can take $\mathcal S$ to be $\Phi_n(m)$, the set of all classes in $\Phi_n$ whose minima are at least $m$, or the set of classes of lattices of certain rank which do not represent a particular root lattice or an additively indecomposable lattice.  Using the results from the previous sections, we determine the $n$-exceptional set of $\mathfrak N$ which is usually infinite.  However, due to the special nature of the chosen $\mathcal S$, $\mathcal E_n(\mathfrak N)$ is often a parametrized family $\{\cls(P(k)) : k \in \mathbb N\}$.   The final step of each case would be considering a lattice $G_{m,n}$ of the form $\mathfrak N\perp P(k_1)\perp \cdots \perp P(k_t)$ for some suitable positive integers $k_1, \ldots, k_t$ depending on $m$ and proving that the $n$-exceptional set of $\mathfrak N\perp P(k_1)\perp \cdots \perp P(k_t)$ has size $m$.

\subsection{$n = 1$ or $2$}

Let $\mathfrak L_{m,n}$ be a universal hull for $\Phi_n(m)$.   The minimum of $\mathfrak L_{m,n}$ must also be $m$ because all lattices in any $\Phi_n(m)$-universality criterion set have minima at least $m$ and $\langle m \rangle \perp \cdots \perp \langle m \rangle$ ($n$ copies of $\langle m \rangle$) must be represented by $\mathfrak L_{m,n}$.  Note that a universal hull for $\Phi_n(m)$ is also $\Phi_\ell(m)$-universal for $1 \leq \ell \leq n$.

In the case of $n = 1$, it is clear that $\mathcal E_1(\mathfrak L_{{m+1},1}) = \{1, \ldots, m\}$.

As for the case $n = 2$,  let
$$G_{m,2} := \langle 1, 2 \rangle \perp \begin{pmatrix}2&1\\1&m+1\end{pmatrix} \perp \mathfrak L_{3,2}.$$
By definition, all unary integral lattices $\langle d \rangle$ with $d \geq 3$ and all the binary integral lattices with minima $\geq 3$ are represented by $\mathfrak L_{3,2}$.  Thus,  $\mathcal E_2(G_{m,2})$ has exactly $m$ classes whose representatives are, respectively,
$$\begin{pmatrix} 2&1\\1&1\end{pmatrix} \cong \langle 1,1\rangle, \,\, \begin{pmatrix} 2&1\\1&2\end{pmatrix}, \dots, \begin{pmatrix} 2&1\\1&m\end{pmatrix}.$$

\subsection{$n = 3, 4$, and $5$}

For any $n$, let  $\mathfrak A_n$ be a universal hull for the set of classes of integral lattices of rank $n + 1$ that {\em do not} represent $A_n$.

We start with the case $n = 3$.  Any integral lattice of rank 2 that is not represented by $\mathfrak A_2$ must be isometric to one of the lattices in \eqref{a2}.  Let
$$G_{m,3}: = I_2 \perp \oline{\mathfrak A}_2 \perp A_2\, (9(m+1) -6)\left[1\, \frac{1}{3}\right].$$
For any $a \geq 2$, $\langle a, a, a \rangle$ does not represent $A_2$, hence it must be represented by $I_2 \perp \oline{\mathfrak A}_2$.  This shows that $A_2 \perp \langle a \rangle$ is represented by $G_{m,3}$.   It is also clear that $A_2\, (9b - 6)\left[1\,\frac{1}{3}\right]$ is represented by $\langle 1 \rangle \perp A_2\, (9(m+1) -6)\left[1\, \frac{1}{3}\right]$ for all $b \geq m + 1$.
Moreover, the image of any representation of $A_2$ into $G_{m,3}$ must sit inside the component $A_2\, (9(m+1) - 6)\left[1\, \frac{1}{3}\right]$ and hence $G_{m,3}$ cannot represent $A_2\, (9b - 6)\left[1\frac{1}{3}\right]$ for any positive integer $b \leq m$.   Therefore,  $\mathcal E_3(G_{m,3})$  has size $m$ and contains the $m$ classes of lattices whose representatives are $A_2\, (9a - 6)\left[1\, \frac{1}{3}\right]$, $1 \leq a \leq m$, respectively.

For $n = 4$, we first note that $\mathfrak L_{2,4}$ represents all integral lattices of rank at most 4 and minimum at least 2.  Thus, $\mathcal E_4(I_3\perp \mathfrak L_{2,4})$ contains only the single class $\cls(I_4)$.  So, we may assume that $m \geq 2$ in the following discussion.  The lattices that are not represented by $\mathfrak A_3$ must be isometric to one of the lattices in \eqref{a3}.

Since $I_3$ represents $A_3$, any quaternary integral lattice not representing $A_3$ must be of the form $I_k \perp N'$ with $k \leq 2$ and $\min(N') \geq 2$. Therefore,  $I_2 \perp \oline{\mathfrak A_3}$ also represents all quaternary integral lattices that do not represent $A_3$.  In particular, it represents $I_2 \perp \langle b, b \rangle$ for any $b \geq 2$ and so $I_3 \perp \langle b \rangle$ is represented by $I_3 \perp \oline{\mathfrak A_3}$.  Since $A_3 \perp \langle b \rangle$ and $A_3\, 4b\left[2\, \frac{1}{2}\right]$ are represented by $I_3 \perp \langle b \rangle$,  the lattices that are not represented by $I_3 \perp \oline{\mathfrak A_3}$ must be isometric to
$$I_4, \quad D_4, \quad A_3 \perp \langle 1 \rangle, \quad \mbox{ or } \quad A_3\, (16b - 12)\left[1\, \frac{1}{4}\right].$$

For any $m \geq 2$, let
$$G_{m,4}: = I_3 \perp \oline{\mathfrak A_3} \perp A_3\, (16m -12)\left[1\, \frac{1}{4}\right].$$
Then $\mathcal E_4(G_{2,4}) = \{\cls(I_4), \cls(D_4) \}$, and for any $m \geq 3$, $\mathcal E_4(G_{m, 4})$ has size $m$ and contains the $m$ classes whose representatives are, respectively,
$$I_4, \quad D_4, \quad \mbox{ or } \quad A_3\, (16b - 12)\left[1\, \frac{1}{4}\right], \quad 2 \leq b \leq m - 1.$$

For the final case of 5-exceptional sets, we consider the lattice
$$G(c,d) := I_4 \perp \oline{\mathfrak A_4} \perp M_c\perp K_d,$$
where $M_c$ and $K_d$ are from \eqref{a4}.  The integral lattices that are not represented by $G(c, d)$ are isometric to some of those lattices in \eqref{a4}.

For any $b \geq 1$, since $\langle b \rangle$ is represented by $I_4$, $A_4 \perp \langle b \rangle$  is represented by $G(c, d)$. It is also clear that for any $c' \geq c$ and $d' \geq d$,  $M_{c'}$ and $K_{d'}$ are represented by $I_4\perp M_c$ and $I_4\perp K_d$, respectively.  So, by Lemma \ref{a4rep}, the lattices that are not represented by $G(c,d)$ are isometric to
$$\left\{\begin{array}{lll}
M_s  \,(1\le s\le 4d-1), & K_t \, (1\le t \le d-1) & \text{if $d\le \frac c4$}, \\
M_s  \,(1\le s\le c-1), & K_t \, (1\le t \le d-1) &\text{if $\frac c4<d<4c-3$}, \\
M_s  \, (1\le s \le c-1),& K_t \, (1\le t\le 4c-2) &\text{if $4c-3\le d$}.\\
\end{array}\right.$$
In particular, if $m=5k+u$ for some positive integer $k$ and an integer $u$ such that $0\le u\le 4$, we have
$$\mathcal E_5(G(4k,k+u+2))=\left\{ \cls(M_s), \cls(K_t) : 1\le s\le 4k-1, 1\le t \le k+u+1\right\}$$
so that  $\left\vert \mathcal E_5(G(4k,k+u+2)) \right\vert = 5k+u=m$.  For $1 \leq m \leq 4$, it is easy to see that $\mathcal E_5(G(2,1)) = \{\cls(M_1)\}$ and
$$\mathcal E_5(G(2,m))=\{\cls(M_1), \cls(K_1),\ldots, \cls(K_{m-1})\} \mbox{ for $2 \le m\le 4$}.$$

\subsection{$n \geq 6$ and $n  \not \equiv 1$ mod $4$}

By Lemma \ref{andn}, if an integral lattice $L$ of rank $n$ primitively represents $D_{n-1}$, then $L$ must be isometric to
\begin{equation} \label{dnp}
D_{n-1}\perp \langle a \rangle, \quad D_{n-1}\, (16a - 4n + 4)\left[1\, \frac{1}{4}\right], \quad \mbox{ or } \quad D_{n-1}\, 4a\left[2\, \frac{1}{2}\right]
\end{equation}
for some suitable positive integer $a$.   However, if $L$ represents $D_{n-1}$ imprimitively, then $L$ must be isometric to
\begin{equation} \label{dn}
I_{n-1} \perp \langle a \rangle \quad \mbox{ or } \quad D_{n-1}[ 1]\perp \langle a \rangle \quad (\mbox{only if $n \equiv 1$ mod 4})
\end{equation}
for some suitable positive integer $a$.  Note that $D_{n-1} \perp \langle a \rangle$ is represented by $I_{n-1}\perp \langle a \rangle$ and $D_{n-1}[ 1]\perp \langle a \rangle$, and that $D_{n-1}\, 4a \left[2\, \frac{1}{2}\right]$ is represented by $I_{n-1}\perp \langle a \rangle$.

From now on, $n$ is assumed to be $\not \not \equiv 1$ mod 4.  Let $\mathfrak D_{n-1}$ be a universal hull of the set of classes of integral lattices of rank $n$ which {\em do not} represent $D_{n-1}$.  An integral lattice of rank $n$ that is not represented by $I_n\perp \oline{\mathfrak D}_{n-1}$ must be isometric to one of those in \eqref{dnp} or \eqref{dn}.  It is clear that $I_n\perp \oline{\mathfrak D}_{n-1}$ represents $I_{n}$.

Since $I_{n-2} \perp \overline{\mathfrak D}_{n-1}$ represents all integral lattices of rank $n$ that do not represent $D_{n-1}$, $I_{n-2}\perp \langle a, a \rangle$ is represented by $I_{n-2} \perp \oline{\mathfrak D}_{n-1}$ for any $a \geq 2$.  This implies that $\langle a, a \rangle$ is represented by $\oline{\mathfrak D}_{n-1}$.  Therefore,  $I_n\perp \oline{\mathfrak D}_{n-1}$ also represents $I_{n-1}\perp \langle a \rangle$ for any $a \geq 2$.

We now claim that for any $a > \frac{n-1}{4}$, $D_{n-1}\, (16a-4n + 4)\left[1\, \frac{1}{4}\right]$ is not represented by $I_n\perp \oline{\mathfrak D}_{n-1}$.  Suppose that $\sigma: D_{n-1}\, (16a-4n + 4)\left[1\, \frac{1}{4}\right] \longrightarrow I_n\perp \oline{\mathfrak D}_{n-1}$ is a representation.  Since $D_{n-1}$ is an indecomposable root lattice, $\sigma(D_{n-1})$ must be inside of $I_{n}$.  If $\sigma(D_{n-1})$ is a primitive sublattice of $I_{n}$, then $I_{n}$ is isometric to one of the lattices in \eqref{dnp} which is not possible.  Therefore, $\q\,\sigma(D_{n-1}) \cap I_{n}$ is a sublattice of $I_{n}$ which properly contains $\sigma(D_{n-1})$.  By considering the discriminants, we see that this sublattice of $I_{n}$ must be isometric to $I_{n-1}$.  So, there is a vector $\by \in \q\sigma(D_{n-1}) \cap I_n$ such that $\sigma^{-1}(\by)$ is the glue vector $\lglue 2\rglue$ of $D_{n-1}$.  Let $\bz$ be a vector which generates the orthogonal complement of $D_{n-1}$ in $D_{n-1}\, (16a-4n + 4)\left[1\, \frac{1}{4}\right]$.  Then $\sigma(\lglue 1\rglue + \frac{\bz}{4})$ is in $I_n\perp \oline{\mathfrak D}_{n-1}$ and
$$\frac{1}{2} = B(\lglue 2\rglue, \lglue 1\rglue) = B\left(\lglue 2\rglue, \lglue 1\rglue + \frac{z}{4}\right) = B\left(\by, \sigma\left(\lglue 1\rglue  + \frac{\bz}{4}\right)\right) \in \z,$$
which is absurd.

So, the integral lattices that are not represented by $I_n\perp \oline{\mathfrak D}_{n-1}$ are precisely those isometric to $D_{n-1}\, (16a - 4n + 4)\left[1\, \frac14\right]$ for all $a > \frac{n-1}{4}$.  Let
$$G_{m,n}: = I_n\perp \oline{\mathfrak D}_{n-1} \perp D_{n-1}\, \left(16\left(\left\lfloor \frac{n-1}{4} \right\rfloor + m + 1\right) - 4n + 4 \right)\left[1\, \frac{1}{4}\right].$$
Then $\mathcal E_n(G_{m,n})$ has size $m$ and contains the classes whose representatives are $D_{n-1}\, (16a - 4n + 4)\left[1\, \frac{1}{4}\right]$, $\lfloor \frac{n-1}{4}\rfloor + 1 \leq a \leq \lfloor \frac{n-1}{4} \rfloor + m$.

\subsection{$n = 9$}

By Lemma \ref{andn} and Lemma \ref{a8rep}, any integral lattice of rank 9 which represents $A_8$  must be isometric one of the following lattices:
$$E_8\perp \langle a \rangle, \quad A_8 \perp \langle a \rangle, \quad \mbox{ and }\quad   A(k,i) := A_8\, 9(9k + i^2)\left[i\, \frac{1}{9}\right],$$
where $a \geq 1$, $1 \leq i \leq 4$, and $k$ is an integer such that $9k + i^2 \geq 1$ (which means that $k \geq 0$, except that $k$ could be $-1$ when $i = 4$).

Note that $A(0,1) \cong I_9$, $A(1,1) \cong A_9$, and $A(0, 2) = D_9$.   Moreover, $A(k,3)$ is represented by $E_8 \perp \langle k + 1 \rangle$ because adjoining the glue vector $\lglue 3 \rglue$ to $A_8$ generates $E_8$.

Let $\overline{\mathfrak A}_8$ be a universal hull of the set of classes of integral lattices of rank 9 that do not represent $A_8$, and let
$$H: = I_8 \perp \oline{\mathfrak A}_8 \perp E_8.$$
It is easy to see that $H$ represents $E_8 \perp \langle a \rangle$ and $A_8 \perp \langle a \rangle$ for any $a \geq 1$.  This means that $H$ also represents $A(k,3)$ for all possible $k$.

We claim that $H$ does not represent $A(k,i)$ for $i = 1, 2$, or 4.  Assume on the contrary that there exists a representation $\sigma: A(k,i) \longrightarrow H$.  Then $\sigma(A_8) \subseteq E_8$ because $A_8$ is an indecomposable root lattice not represented by $I_8 \perp \oline{\mathfrak A}_8$.  If $\bz$ is a vector which generates the orthogonal complement of $A_8$ in $A(k, i)$, then $\sigma(\lglue i \rglue + \frac{\bz}{9}) \in H = E_8 \perp E_8^\perp$ and hence $\sigma(\lglue i \rglue) \in E_8$. This is impossible since $Q(\lglue i \rglue) \not \in \z$ for $i = 1, 2$, or 4.

We first treat the case $m \geq 21$.  Suppose that $m = 21q + r$, where $q \geq 1$ and $0 \leq r \leq 20$.  Let
$$G_{m,9}: = A(q +1, 1) \perp A(16q +r - 6, 4) \perp H.$$

It is clear that $A(q + 1, 1) \perp I_8$, hence $G_{m,9}$ as well,  represents $A(k, 1)$ for all $k \geq q + 1$.  By the same token,
$G_{m,9}$ represents $A(k', 4)$ for all $k ' \geq 16q + r - 6$.  Since $A(k, 1)$ represents $A(4k, 2)$ by Lemma \ref{anrep}, therefore $G_{m,9}$ represents $A(\tilde{k}, 2)$ for all $\tilde{k} \geq 4q + 4$.

Suppose that, for some $k \geq 0$,  $\sigma: A(k, 1) \longrightarrow G_{m,9}$ is a representation.  Then $\sigma(A_8)$ must be inside of either $A(q +1, 1)$ or $A(16q + r - 6, 4)$.  Let $\bz$ be a vector which generates the orthogonal complement of $A_8$ in $A(k, 1)$.  Assume that $\sigma(A_8) \subseteq A(q + 1, 1)$.  Since $Q(\lglue 1 \rglue) \not \in \z$, $\sigma(\bz)$, which is in the orthogonal complement of $A(q + 1, 1)$ in $G_{m,9}$, cannot be inside of $A(16q + r - 6, 4) \perp H$.  Thus, $k$ must be at least $q + 1$ in this case.  If $\sigma(A_8) \subseteq A(16q + r - 6, 4)$, then a similar reasoning together with Lemma \ref{anrep} show that $k \geq 4(16q + r - 6) + 7 \geq q + 1$.  Therefore, $k$ must be at least $q + 1$.  Using the same line of argument one can show that the following integral lattices are representatives of the classes in $\mathcal E_9(G_{m,9})$:
$$\left\{
\begin{array}{ll}
A(k, 1) & 0 \leq k \leq q,\\
A(\tilde{k}, 2) & 0 \leq \tilde{k} \leq 4q + 3,\\
A(k', 4) & -1 \leq k' \leq 16q + r - 7.
\end{array} \right.$$
Therefore, $\mathcal E_9(G_{m,9})$ contains exactly
$$(q + 1) + (4q + 4) + (16q + r - 5) = 21q + r = m$$
classes.  This completes the proof for the case $m \geq 21$.

For $20 \geq m \geq 6$, we consider
$$G_{m,9}: = A(1,1) \perp A(m - 6, 4) \perp G$$
whose $9$-exceptional set is
$$\left\{ \cls(A(0,1)), \cls(A(\tilde{k}, 2)) \, (0 \leq \tilde{k} \leq 3), \cls(A(k', 4)) \, (-1 \leq k' \leq m - 7)\right\}.$$
For $1 \leq  m \leq 5$, let
$$G_{m,9}: = A(1,1) \perp A(m-1, 2) \perp A(-1, 4)\perp G$$
whose $9$-exceptional set is
$$\left\{ \cls(A(0,1)), \cls(A(\tilde{k}, 2))\, (0 \leq \tilde{k} \leq m-2) \right\}.$$

In all cases, $\left \vert \mathcal E_9(G_{m,9}) \right \vert = m$.

\subsection{$n \geq 13$ and $n \equiv 1$ mod $4$}

The proof will be divided into three cases: $n = 13$, $n = 17$, and $n \geq 21$.

Let us consider the case $n = 13$ first.  Recall that the integral lattice $L_{12} = E_7\, A_5[1\, 3]$ from Proposition \ref{addindecomposable1} is additively indecomposable.  Let $\mathfrak N$ be a universal hull of  the set of classes of integral lattices of rank 13 that do not represent $L_{12}$, and let
$$G_{m,13}: = L_{12}\, (9m + 6)\left[\bv\, \frac{1}{3}\right] \perp \mathfrak N,$$
where $\bv$ is the minimal vector of $L_{12}^\#$ given in Proposition \ref{addindecomposable1}.

For any integer $a \geq 1$, the lattice $\langle a \rangle \perp \cdots \perp \langle a \rangle$ of rank 13 does not represent $L_{12}$ because $L_{12}$ is additively indecomposable.  Therefore, it must be represented by $\mathfrak N$ and hence $L_{12} \perp \langle a \rangle$ is represented by $G_{m,13}$.  By a similar argument, for all $k \geq m$ the lattices $L_{12}\, (9k + 6) \left[\bv\, \frac{1}{3} \right]$ are represented by $G_{m,13}$.

Now, suppose that $\sigma: L_{12}\, (9k + 6)\left[\bv\, \frac{1}{3}\right] \longrightarrow G_{m,13}$ is a representation.  Since $L_{12}$ is additively indecomposable, $\sigma$ must send $L_{12}$ into the orthogonal summand $L_{12}\, (9m + 6)\left[\bv\, \frac{1}{3}\right]$ of $G_{m,13}$.  Moreover, since $Q(\bv) = \frac{4}{3} \not \in \z$,  the orthogonal complement of $L_{12}$ in $L_{12}\, (9k + 6)\left[\bv\, \frac{1}{3}\right]$ cannot be sent by $\sigma$ into $\mathfrak N$.  So, $k$ must be at least $m$ and by Proposition \ref{addindecomposable1}, the $13$-exceptional set of $G_{m,13}$ contains exactly $m$ classes whose representatives are
$$L_{12}\, 6 \left[\bv\, \frac{1}{3}\right], \ldots, L_{12}\, (9m - 3)\left[\bv\, \frac{1}{3} \right].$$
This proves the case $n = 13$.  The cases $m = 17$ and $m \geq 21$ can be proved by the same argument using instead Proposition \ref{addindecomposable2} and Proposition \ref{mn}, respectively.  We leave the details to the readers.

\end{document}